\documentclass[oneside,french,english]{amsart}
\usepackage[T1]{fontenc}
\usepackage[latin9]{inputenc}
\usepackage{geometry}
\usepackage{textcomp}
\usepackage{amstext}
\usepackage{amsthm}
\usepackage{amssymb}

\makeatletter
\numberwithin{equation}{section}
\numberwithin{figure}{section}
  \theoremstyle{plain}
  \newtheorem*{thm*}{\protect\theoremname}
\theoremstyle{plain}
\newtheorem{thm}{\protect\theoremname}
  \theoremstyle{definition}
  \newtheorem{defn}[thm]{\protect\definitionname}
  \theoremstyle{plain}
  \newtheorem{lem}[thm]{\protect\lemmaname}
  \theoremstyle{plain}
  \newtheorem{prop}[thm]{\protect\propositionname}
  \theoremstyle{plain}
  \newtheorem{cor}[thm]{\protect\corollaryname}
  \theoremstyle{remark}
  \newtheorem{rem}[thm]{\protect\remarkname}
  \theoremstyle{definition}
  \newtheorem{example}[thm]{\protect\examplename}

\usepackage{xypic}

\makeatother

\usepackage{babel}
\makeatletter
\addto\extrasfrench{%
   \providecommand{\fg}{\ifdim\lastskip>\z@\unskip\fi~\frqq}%
}

\makeatother
  \addto\captionsenglish{\renewcommand{\corollaryname}{Corollary}}
  \addto\captionsenglish{\renewcommand{\definitionname}{Definition}}
  \addto\captionsenglish{\renewcommand{\examplename}{Example}}
  \addto\captionsenglish{\renewcommand{\lemmaname}{Lemma}}
  \addto\captionsenglish{\renewcommand{\propositionname}{Proposition}}
  \addto\captionsenglish{\renewcommand{\remarkname}{Remark}}
  \addto\captionsenglish{\renewcommand{\theoremname}{Theorem}}
  \addto\captionsfrench{\renewcommand{\corollaryname}{Corollaire}}
  \addto\captionsfrench{\renewcommand{\definitionname}{Définition}}
  \addto\captionsfrench{\renewcommand{\examplename}{Exemple}}
  \addto\captionsfrench{\renewcommand{\lemmaname}{Lemme}}
  \addto\captionsfrench{\renewcommand{\propositionname}{Proposition}}
  \addto\captionsfrench{\renewcommand{\remarkname}{Remarque}}
  \addto\captionsfrench{\renewcommand{\theoremname}{Théorème}}
  \providecommand{\corollaryname}{Corollary}
  \providecommand{\definitionname}{Definition}
  \providecommand{\examplename}{Example}
  \providecommand{\lemmaname}{Lemma}
  \providecommand{\propositionname}{Proposition}
  \providecommand{\remarkname}{Remark}
  \providecommand{\theoremname}{Theorem}
\providecommand{\theoremname}{Theorem}

\begin{document}

\author{Adrien Dubouloz}

\address{IMB UMR5584, CNRS, Univ. Bourgogne Franche-Comté, F-21000 Dijon,
France.}

\email{adrien.dubouloz@u-bourgogne.fr}

\title{$\mathbb{A}^{2}$-fibrations between affine spaces are trivial $\mathbb{A}^{2}$-bundles }
\begin{abstract}
We give a criterion for a flat fibration with affine plane fibers
over a smooth scheme defined over a field of characteristic zero to
be a Zariski locally trivial $\mathbb{A}^{2}$-bundle. An application
is a positive answer to a version of the Dolgachev-Weisfeiler Conjecture
for such fibrations: a flat fibration $\mathbb{A}^{m}\rightarrow\mathbb{A}^{n}$
with all fibers isomorphic to $\mathbb{A}^{2}$ is the trivial $\mathbb{A}^{2}$-bundle. 
\end{abstract}

\maketitle

\section*{Introduction }

An $\mathbb{A}^{n}$-fibration over a scheme $X$ is a flat affine
morphism of finite presentation $\pi:V\rightarrow X$ whose fibers,
closed or not, are all isomorphic to the affine $n$-space $\mathbb{A}^{n}$
over the corresponding residue fields. A version of the Dolgachev-Weisfeiler
Conjecture \cite[3.8.3]{DoW75} (see also \cite[Conjecture 3.14]{To16})
asks whether an $\mathbb{A}^{n}$-fibration $\pi:V\rightarrow X$
over a normal locally noetherian integral scheme $X$ is a form of
the affine $n$-space $\mathbb{A}_{X}^{n}=X\times\mathbb{A}^{n}$
over $X$ for the Zariski topology, or at least for the \'etale topology.
The conjecture in such general form remains open, and so far, only
the following two special cases are known (see \cite{BhD94} for a
survey): 

1) For $n=1$, every $\mathbb{A}^{1}$-fibration $\pi:V\rightarrow X$
over a normal locally noetherian integral scheme $X$ is a Zariski
locally trivial $\mathbb{A}^{1}$-bundle by successive results of
Kambayashi-Miyanishi \cite{KM78} and Kambayashi-Wright \cite{KW85}. 

2) For $n=2$, a result of Sathaye \cite{Sa83} asserts that an $\mathbb{A}^{2}$-fibration
$\pi:V\rightarrow X$ over the spectrum $X$ of a rank one discrete
valuation ring containing $\mathbb{Q}$ is the trivial $\mathbb{A}^{2}$-bundle
over $X$. The proof depends on the famous Abhyankar-Moh Theorem,
and the characteristic zero hypothesis is crucial as illustrated by
counter-examples in positive characteristic constructed by Asanuma
\cite[\S 5.1]{As87}. Additional results concerning the structure
of $\mathbb{A}^{2}$-fibrations over spectra of one-dimensional noetherian
domains containing $\mathbb{Q}$ have been obtained later on by Asanuma-Bathwadekar
\cite{AsB97}. 

In contrast, the stable structure of general $\mathbb{A}^{n}$-fibrations
is quite well understood: it was established by Asanuma \cite{As87}
that every such fibration $\pi:V\rightarrow X$ over a smooth affine
scheme $X$ defined over a field of characteristic zero is stably
isomorphic to the total space of vector bundle over $X$, in the sense
that $V\times_{X}\mathbb{A}_{X}^{m}\simeq E\times_{X}\mathbb{A}_{X}^{m}$
for some vector bundle $p:E\rightarrow X$ of rank $n$ over $X$.
The question whether the $\mathbb{A}_{X}^{m}$ factor can be ``canceled''
to obtain that these $\mathbb{A}^{n}$-fibrations are themselves vector
bundles remains open in general. 

On the other hand, by a result of Bass-Connell-Wright \cite{BCW77},
a Zariski locally trivial $\mathbb{A}^{n}$-bundle $\pi:V\rightarrow X$
over an affine scheme always carries the structure of a vector bundle:
there exist local trivializations of $V$ on a Zariski open cover
of $X$ for which the corresponding transition isomorphisms are linear
automorphisms of $\mathbb{A}^{n}$. A well-known property of vector
bundles, and more generally of affine-linear bundles $\nu:V\rightarrow X$
- that is, locally trivial $\mathbb{A}^{n}$-bundles whose transition
isomorphisms are affine automorphisms of $\mathbb{A}^{n}$- is that
their relative cotangent sheaves $\Omega_{V/X}^{1}$ are induced from
$X$, i.e. isomorphic to the pull-back to $V$ of a locally free sheaf
of $\mathcal{O}_{X}$-modules. Summing up, if an $\mathbb{A}^{n}$-fibration
$\pi:V\rightarrow X$ over an affine scheme $X$ is a Zariski locally
trivial $\mathbb{A}^{n}$-bundle, then its relative cotangent sheaf
is induced from $X$. Our main result, which can be summarized as
follows, implies in particular that the converse holds for $\mathbb{A}^{2}$-fibrations
over smooth affine schemes:
\begin{thm*}
Let $\pi:V\rightarrow X$ be an $\mathbb{A}^{2}$-fibration over a
smooth locally noetherian scheme defined over a field of characteristic
zero. If $\Omega_{V/X}^{1}$ is induced from $X$ then $\pi:V\rightarrow X$
is an affine-linear bundle. 
\end{thm*}
Note that combined with the fact that vector bundles on $\mathbb{A}_{k}^{m}$
are trivial by the Quillen-Suslin Theorem \cite{Qu76,Su76}, this
characterization implies that an $\mathbb{A}^{2}$-fibration $\pi:\mathbb{A}_{k}^{m}\rightarrow\mathbb{A}_{k}^{n}$
is isomorphic to the trivial $\mathbb{A}^{2}$-bundle $\mathbb{A}_{k}^{n}\times\mathbb{A}_{k}^{2}$. 

As an application in the special case $m=4$, we deduce that the famous
Vénéreau polynomials \cite{Ve01}, as well as large collections of
``Vénéreau-type'' polynomials introduced by Daigle-Freudenburg \cite{DaFr10}
and Lewis \cite{Le13}, are variables of polynomials rings in four
variables over a field (see $\S$ \ref{subsec:Venereau}). As another
application, we answer in $\S$ \ref{subsec:Freudenburg} a question
raised by Freudenburg \cite{Fr09} concerning the structure of locally
nilpotent derivations with a slice on polynomial rings in three variables
over a base ring. 

\section{Recollection on $\mathbb{A}^{n}$-fibrations, $\mathbb{A}^{n}$-bundles
and affine-linear bundles}

In what follows we fix a base field $k$ of characteristic zero. All
schemes considered are defined over $k$. 
\begin{defn}
Let $X$ be a scheme. A flat affine morphism of finite presentation
$f:V\rightarrow X$ is called: 

1) An \emph{$\mathbb{A}^{n}$-fibration} if for every point $x\in X$,
the scheme theoretic fiber $f^{-1}(x)$ is isomorphic to the affine
$n$-space $\mathbb{A}_{\kappa(x)}^{n}$ over the residue field $\kappa(x)$. 

2) A \emph{Zariski $($resp. \'etale$)$ locally trivial $\mathbb{A}^{n}$-bundle}
if there exists a Zariski open (resp. \'etale) cover $Y\rightarrow X$
of $X$ such that $V\times_{X}Y$ is isomorphic as a scheme over $Y$
to the trivial $\mathbb{A}^{n}$-bundle $\mathbb{A}_{Y}^{n}=Y\times\mathbb{A}_{k}^{n}$. 
\end{defn}
Elementary examples of Zariski locally trivial $\mathbb{A}^{n}$-bundles
are vector bundles. To fix a convention, by a\emph{ }vector bundle
of rank $n\geq1$ over a scheme $X$, we mean the relative spectrum
$p:E={\rm Spec}({\rm Sym}^{\cdot}\mathcal{E}^{\vee})\rightarrow X$
of the symmetric algebra of the dual of a locally free $\mathcal{O}_{X}$-module
$\mathcal{E}$ of rank $n$. Recall that every vector bundle $p:E\rightarrow X$
carries the structure of a Zariski locally constant group scheme for
the law given by the addition of germs of sections. An \emph{$E$-torsor
}is an \'etale locally trivial principal homogeneous $E$-bundle,
that is, a scheme $\nu:V\rightarrow X$ equipped with an action $\mu:E\times_{X}V\rightarrow V$
of $E$ for which there exists an \'etale cover $Y\rightarrow X$
such that $V\times_{X}Y$ is equivariantly isomorphic to $E\times_{X}Y$
acting on itself by translations. 
\begin{defn}
An\emph{ affine-linear bundle} or rank $n\geq1$ over a scheme $X$
is \'etale locally trivial $\mathbb{A}^{n}$-bundle $\nu:V\rightarrow X$
which can be further equipped with the structure of an $E$-torsor
for a suitable vector bundle $p:E\rightarrow X$ of rank $n$ over
$X$.
\end{defn}
Equivalently, an affine-linear bundle is an \'etale locally trivial
$\mathbb{A}^{n}$-bundle $\nu:V\rightarrow X$ for which there exists
an \'etale cover $f:Y\rightarrow X$ and an isomorphism $V\times_{X}Y\stackrel{\varphi}{\rightarrow}\mathbb{A}_{Y}^{n}$
such that over $Y\times_{X}Y$ equipped with the two projections $\mathrm{pr}_{1},\mathrm{pr}_{2}:Y\times_{X}Y\rightarrow Y$,
$\text{\ensuremath{\mathrm{pr}}}_{2}^{*}\varphi\circ\text{\ensuremath{\mathrm{pr}}}_{1}^{*}\varphi^{-1}$
is an affine automorphism of $\mathbb{A}_{Y\times Y}^{n}$ , i.e.
is given by an element $(A,T)$ of $\mathrm{Aff}_{n}(Y\times_{X}Y)={\rm GL}_{n}(Y\times_{X}Y)\rtimes\mathbb{G}_{a}^{r}(Y\times_{X}Y)$.
The vector bundle $E$ for which $\nu:V\rightarrow X$ is an $E$-torsor
is uniquely determined up to isomorphism by the fact that its class
in $H_{\textrm{ét}}^{1}(X,{\rm GL}_{n})$ coincides with that of the
$1$-cocyle $A\in{\rm GL}_{n}(Y\times_{X}Y)$ for the \'etale cover
$f:Y\rightarrow X$ of $X$. Since the affine group $\mathrm{Aff}_{n}=\mathrm{GL}_{n}\ltimes\mathbb{G}_{a}^{n}$
is special \cite{Gr58}, every affine-linear bundle is actually locally
trivial in the Zariski topology. Furthermore, there is a one-to-one
correspondence between isomorphy classes of $E$-torsors over $X$
and elements of the cohomology group $\check{H}^{1}(X,E)\simeq H_{\mathrm{Zar}}^{1}(X,E)\simeq H_{\mathrm{\textrm{ét}}}^{1}(X,E)$,
with $0\in\check{H}^{1}(X,E)$ corresponding the the trivial $E$-torsor
$p:E\rightarrow X$ (see e.g. \cite[XI.4]{SGA1}). In particular,
every $E$-torsor over an affine scheme $X$ is isomorphic to the
trivial one . \\

Recall \cite[16.4.9]{EGAIV} that given a vector bundle $p:E=\mathrm{Spec}(\mathrm{Sym}^{\cdot}\mathcal{E}^{\vee})\rightarrow X$,
there exists a canonical isomorphism of $\mathcal{O}_{E}$-modules
$p^{*}\mathcal{E}^{\vee}\stackrel{\simeq}{\rightarrow}\Omega_{E/X}^{1}$
defined as the composition of the canonical homomorphism $p^{*}\mathcal{E}^{\vee}\rightarrow\mathcal{O}_{E}$
with the canonical derivation $d_{E/X}:\mathcal{O}_{E}\rightarrow\Omega_{E/X}^{1}$.
A direct local calculation shows more generally that the relative
cotangent sheaf $\Omega_{V/X}^{1}$ of any $E$-torsor $\nu:V\rightarrow X$
is isomorphic to $\nu^{*}\mathcal{E}^{\vee}$. This property actually
characterizes affine-linear bundles among \'etale locally trivial
$\mathbb{A}^{n}$-bundles: 
\begin{lem}
\label{lem:AffLin-charac} A \'etale locally trivial $\mathbb{A}^{n}$-bundle
$\pi:V\rightarrow X$ over a scheme $X$ is an affine-linear bundle
if and only if there exists a locally free sheaf $\mathcal{E}$ of
rank $n$ on $X$ such that $\Omega_{V/X}^{1}\simeq\pi^{*}\mathcal{E}^{\vee}$.
If such a locally free sheaf $\mathcal{E}$ exists, then $\pi:V\rightarrow X$
is torsor under the rank $n$ vector bundle $p:E=\mathrm{Spec}(\mathrm{Sym}^{\cdot}\mathcal{E}^{\vee})\rightarrow X$
on $X$. 
\end{lem}
\begin{proof}
Since $\pi$ is an affine morphism of finite type, $\mathcal{A}=\pi_{*}\mathcal{O}_{V}$
is a quasi-coherent $\mathcal{O}_{X}$-algebra of finite type. Let
$d_{V/X}:\mathcal{O}_{V}\rightarrow\Omega_{V/X}^{1}$ be the canonical
$\mathcal{O}_{X}$-derivation. Since $\Omega_{V/X}^{1}\simeq\pi^{*}\mathcal{E}^{\vee}$,
$\pi_{*}\Omega_{V/X}^{1}$ is isomorphic to $\pi_{*}\mathcal{O}_{V}\otimes\mathcal{E}^{\vee}=\mathcal{A}\otimes\mathcal{E}^{\vee}$
by the projection formula, and the direct image $\pi_{*}d_{V/X}$
of $d_{V/X}$ is an $\mathcal{O}_{X}$-derivation $\partial_{1}:\mathcal{A}\rightarrow\mathcal{A}\otimes\mathcal{E}^{\vee}$
of $\mathcal{A}$ with values in $\mathcal{A}\otimes\mathcal{E}^{\vee}$.
For every $r\geq2$, we let $\partial_{r}:\mathcal{A}\rightarrow\mathcal{A}\otimes\mathrm{Sym}^{r}\mathcal{E}^{\vee}$
be the $\mathcal{O}_{X}$-linear homomorphism defined as the composition
of 
\[
(\partial_{1}\otimes\mathrm{id}_{\mathrm{Sym}^{r-1}\mathcal{E}^{\vee}})\circ{\displaystyle \partial_{r-1}:}\mathcal{A}\rightarrow\mathcal{A}\otimes\mathrm{Sym}^{r-1}\mathcal{E}^{\vee}\rightarrow(\mathcal{A}\otimes\mathcal{E}^{\vee})\otimes\mathrm{Sym}^{r-1}\mathcal{E}^{\vee}
\]
with the canonical homomorphism $\mathcal{A}\otimes(\mathcal{E}^{\vee}\otimes\mathrm{Sym}^{r-1}\mathcal{E}^{\vee})\rightarrow\mathcal{A}\otimes\mathrm{Sym}^{r}\mathcal{E}^{\vee}$.
We let $\partial_{0}=\mathrm{id}_{\mathcal{A}}:\mathcal{A}\rightarrow\mathcal{A}\otimes\mathrm{Sym}^{0}\mathcal{E}^{\vee}\simeq\mathcal{A}$.
The kernels $\mathcal{K}\mathrm{er}\partial_{r}$, $r\geq0$, form
an increasing sequence of quasi-coherent sub-$\mathcal{O}_{X}$-modules
of $\mathcal{A}$. We claim that $\mathcal{A}=\mathrm{colim}_{r\geq0}\mathcal{K}er\partial_{r}$
and that the map 
\[
\exp(\partial_{1}):=\sum_{r\geq0}\frac{\partial_{r}}{r!}:\mathcal{A}\rightarrow\bigoplus_{r\geq0}\mathcal{A}\otimes\mathrm{Sym}^{r}\mathcal{E}^{\vee}\simeq\mathcal{A}\otimes\mathrm{Sym}^{\cdot}\mathcal{E}^{\vee}
\]
is a well-defined homomorphism of $\mathcal{O}_{X}$-algebra, corresponding
to an action $\mu:E\times_{X}V\rightarrow V$ of the vector bundle
$p:E=\mathrm{Spec}(\mathrm{Sym}^{\cdot}\mathcal{E}^{\vee})\rightarrow X$
on $\pi:V\rightarrow X$. This can be seen as follows: let $f:Y\rightarrow X$
be an \'etale cover of $X$ on which $\pi:V\rightarrow X$ becomes
trivial, say $W=V\times_{X}Y\simeq\mathbb{A}_{Y}^{n}=\mathrm{Spec}(\mathcal{O}_{Y}[t_{1},\ldots,t_{n}])$.
Then $\Omega_{W/Y}^{1}$ is the free $\mathcal{O}_{W}$-module with
basis $dt_{1},\ldots,dt_{n}$ and since $\Omega_{W/Y}^{1}=\mathrm{pr}_{V}^{*}\Omega_{V/X}^{1}\simeq\mathrm{pr}_{Y}^{*}f^{*}\mathcal{E}^{\vee}$,
we conclude by restricting to the zero section of $\mathbb{A}_{Y}^{n}$
that $f^{*}\mathcal{E}^{\vee}\simeq\mathcal{O}_{Y}^{\oplus n}$. Via
these isomorphisms, the homomorphism 
\[
f^{*}\partial_{r}:f^{*}\mathcal{A}=(\mathrm{pr}_{Y})_{*}\mathcal{O}_{W}\rightarrow f^{*}(\mathcal{A}\otimes\mathrm{Sym}^{r}\mathcal{E}^{\vee})\simeq(\mathrm{pr}_{Y})_{*}\mathcal{O}_{W}\otimes\mathrm{Sym}^{r}f^{*}\mathcal{E}^{\vee}
\]
coincides with the homomorphism of $\mathcal{O}_{Y}$-modules 
\[
\begin{array}{ccc}
d_{r}:\mathcal{O}_{Y}[t_{1},\ldots,t_{n}] & \longrightarrow & \mathcal{O}_{Y}[t_{1},\ldots,t_{n}]\otimes\mathrm{Sym}^{r}(\mathcal{O}_{Y}\langle dt_{1},\ldots,dt_{n}\rangle)\\
p(t_{1},\ldots t_{n}) & \mapsto & {\displaystyle \sum_{I=(i_{1},\ldots,i_{n}),i_{1}+\cdots i_{r}=r}}\sum\frac{\partial^{r}p}{\partial t_{1}^{i_{1}}\cdots\partial t_{n}^{i_{n}}}\otimes dt_{1}^{i_{1}}\cdots dt_{n}^{i_{n}}.
\end{array}
\]
Since colimits commute with pull-backs, we have $f^{*}\mathrm{colim}_{r\geq0}\mathcal{K}er\partial_{r}=\mathrm{colim}_{r\geq0}\mathcal{K}erd_{r}=\mathcal{O}_{Y}[t_{1},\ldots,t_{n}]$.
This implies that the injective homomorphism $\mathrm{colim}_{r\geq0}\mathcal{K}er\partial_{r}\rightarrow\mathcal{A}$
is also surjective, hence that $\exp(\partial_{1})$ is indeed well-defined.
Now we have to check that it satisfies the usual axioms for being
the co-morphism of an action of $E$ on $V$, namely, the commutativity
of the diagrams

\selectlanguage{french}%
\[\xymatrix{\mathcal{A} \ar[r]^{\exp(\partial_1)} \ar[d]_{\exp(\partial_1)} &  \mathcal{A}\otimes \mathrm{Sym}^{\cdot} \mathcal{E}^{\vee} \ar[d]^-{\exp(\partial_1)\otimes \mathrm{id}} & & \mathcal{A} \ar[r]^-{\exp(\partial_1)}  \ar[d]_{\mathrm{id}} & \mathcal{A}\otimes \mathrm{Sym}^{\cdot} \mathcal{E}^{\vee} \ar[d]^{\mathrm{id}\otimes\varepsilon}\\ \mathcal{A}\otimes \mathrm{Sym}^{\cdot} \mathcal{E}^{\vee} \ar[r]^-{\mathrm{id}\otimes\mathrm{m}} & \mathcal{A}\otimes \mathrm{Sym}^{\cdot} \mathcal{E}^{\vee} \otimes \mathrm{Sym}^{\cdot} \mathcal{E}^{\vee}  & & \mathcal{A} \ar[r]^-{\simeq} & \mathcal{A}\otimes\mathcal{O}_X }\]\foreignlanguage{english}{where
$\mathrm{m}:\mathrm{Sym}^{\cdot}\mathcal{E}^{\vee}\rightarrow\mathrm{Sym}^{\cdot}\mathcal{E}^{\vee}\otimes\mathrm{Sym}^{\cdot}\mathcal{E}^{\vee}\simeq\mathrm{Sym}^{\cdot}(\mathcal{E}^{\vee}\oplus\mathcal{E}^{\vee})$
is the homomorphism of $\mathcal{O}_{X}$-algebra induced by the diagonal
homomorphism $\mathcal{E}^{\vee}\rightarrow\mathcal{E}^{\vee}\oplus\mathcal{E}^{\vee}$
and where $\varepsilon:\mathrm{Sym}^{\cdot}\mathcal{E}^{\vee}\rightarrow\mathcal{O}_{X}$
is the canonical homomorphism with kernel $\mathcal{E}^{\vee}\cdot\mathrm{Sym}^{\cdot}\mathcal{E}^{\vee}$.
The commutativity of these diagrams can be checked locally on an \'etale
cover of $X$. But by construction, $f^{*}\exp(\partial_{1}):f^{*}\mathcal{A}\rightarrow f^{*}\mathcal{A}\otimes\mathrm{Sym}^{\cdot}f^{*}\mathcal{E}^{\vee}$
coincides with $\exp(d_{1})$, which is precisely the co-morphism
of the action of $E\times_{X}Y\simeq\mathbb{G}_{a,Y}^{n}$ on $V\times_{X}Y\simeq\mathbb{A}_{Y}^{n}$
by translations. The above diagrams are thus commutative, and we conclude
that $\exp(\partial_{1})$ is the co-morphism of an action $\mu:E\times_{X}V\rightarrow V$
of the vector bundle $p:E=\mathrm{Spec}(\mathrm{Sym}^{\cdot}\mathcal{E}^{\vee})\rightarrow X$
on $\pi:V\rightarrow X$ for which $V\times_{X}Y$ is equivariantly
isomorphic to $E\times_{X}Y$ acting on itself by translations. This
shows that $\pi:V\rightarrow X$ is an $E$-torsor, as desired. }
\end{proof}

\section{A characterization of affine-linear bundles among $\mathbb{A}^{n}$-fibrations }

In this section, we establish a more flexible characterization of
affine-linear bundles among all $\mathbb{A}^{n}$-fibrations instead
of only locally trivial ones. We begin with the following local version.
\begin{prop}
\label{prop:MainProp} Let $B$ be the spectrum of a noetherian normal
local ring of dimension $d\geq2$ and let $\pi:V\rightarrow B$ be
an $\mathbb{A}^{n}$-fibration. Suppose that

a) $B$ is regular or $\pi$ has a section, 

b) $\pi$ is a Zariski locally trivial $\mathbb{A}^{n}$-bundle outside
the closed point of $B$,

c) $\Omega_{V/B}^{1}$ is trivial.

\noindent Then $\pi:V\rightarrow B$ is a trivial $\mathbb{A}^{n}$-bundle. 
\end{prop}
\begin{proof}
Let $b_{0}$ be the closed point of $B$ and let $B_{*}=B\setminus\{b_{0}\}$.
Since $\Omega_{V/B}^{1}\simeq\pi^{*}\mathcal{O}_{B}^{\oplus n}$,
it follows from Lemma \ref{lem:AffLin-charac} that the restriction
$\pi:V_{*}\rightarrow B_{*}$ of $\pi:V\rightarrow B$ over $B_{*}$
is a $\mathbb{G}_{a}^{n}$-torsor. It suffices to show that $\pi:V_{*}\rightarrow B_{*}$
is a trivial $\mathbb{G}_{a}^{n}$-torsor. Indeed, if so, then since
$V$ and $B\times\mathbb{A}^{n}$ are both affine and normal, the
isomorphism $V_{*}\simeq B_{*}\times\mathbb{A}^{n}$ extends to an
isomorphism $V\simeq B\times\mathbb{A}^{n}$ of schemes over $B$.
If $\pi$ has a section, then $\pi:V_{*}\rightarrow B_{*}$ has a
section hence is the trivial $\mathbb{G}_{a}^{n}$-torsor. Now suppose
that $B$ is regular and denote by $g=(g_{1},\ldots,g_{n})$ the isomorphy
class of $\pi:V_{*}\rightarrow B_{*}$ in $H^{1}(B_{*},\mathcal{O}_{B}^{\oplus n})\simeq H^{1}(B_{*},\mathcal{O}_{B})^{\oplus n}$.
Since $B$ is affine, the vanishing of the cohomology groups $H^{i}(B,\mathcal{O}_{B})$,
$i\geq1$, implies that in the long exact excision sequence of cohomology
groups for the pair $(B,b_{0})$ the connecting homomorphism $H^{1}(B_{*},\mathcal{O}_{B})\rightarrow H_{b_{0}}^{2}(B,\mathcal{O}_{B})$
is an isomorphism. If $d\geq3$ then since $B$ is regular, $H_{b_{0}}^{2}(B,\mathcal{O}_{B})=0$
(see e.g. \cite[Theorem 3.8]{Gloc67}) and so, $\pi:V_{*}\rightarrow B_{*}$
is the trivial $\mathbb{G}_{a}^{n}$-torsor. 

It thus remains to consider the case where $B$ is the spectrum of
a noetherian $2$-dimensional regular local ring. Suppose that there
exists a index $i$ such that $g_{i}\neq0$, say $i=1$. Then there
exist a nontrivial $\mathbb{G}_{a}$-torsor $\pi_{1}:W_{1}\rightarrow B_{*}$
with isomorphy class $g_{1}$ and a $\mathbb{G}_{a}^{n-1}$-torsor
$\pi_{2}:W_{2}\rightarrow B_{*}$ with isomorphy class $(g_{2},\ldots,g_{n})$
such that $V_{*}\simeq W_{1}\times_{B_{*}}W_{2}$. Since $\pi_{1}:W_{1}\rightarrow B_{*}$
is nontrivial, the same argument as in the proof of Proposition 1.2
in \cite{DuFin14} shows that $W_{1}$ is an affine scheme. Thus $V_{*}$
is an affine scheme too as the projection $\mathrm{pr}_{1}:V_{*}\rightarrow W_{1}$
is a $\mathbb{G}_{a}^{n-1}$-torsor, hence an affine morphism. But
$V\setminus V_{*}=\pi^{-1}(b_{0})$ has codimension $2$ in $V$,
in contradiction to the fact that the complement of an affine open
subscheme in a locally noetherian scheme has pure codimension $1$.
So $\pi:V_{*}\rightarrow B_{*}$ is a trivial $\mathbb{G}_{a}^{n}$-torsor
as desired. 
\end{proof}
As a consequence of the previous proposition, we obtain the following
characterization:
\begin{cor}
\label{cor:AffLin-in-Fib} Let $\pi:V\rightarrow X$ be an $\mathbb{A}^{n}$-fibration
over a smooth locally noetherian scheme $X$. Suppose that:

1) $\pi$ is a Zariski locally trivial $\mathbb{A}^{n}$-bundle outside
a closed subset $Z\subset X$ of codimension $\geq2$, 

2) $\Omega_{V/B}^{1}=\pi^{*}\mathcal{E}^{\vee}$ for some locally
free sheaf $\mathcal{E}$ of rank $n$ on $X$. 

\noindent Then $\pi:V\rightarrow X$ is a torsor under the vector
bundle $p:E=\mathrm{Spec}(\mathrm{Sym}^{\cdot}\mathcal{E}^{\vee})\rightarrow X$
on $X$. 
\end{cor}
\begin{proof}
Let $U\subset X$ be the largest open subset over which $\pi:V\rightarrow X$
restricts to a Zariski locally trivial $\mathbb{A}^{n}$-bundle. In
view of Lemma \ref{lem:AffLin-charac}, it suffices to show that $U=X$.
By hypothesis, $U$ contains all codimension $1$ points of $X$.
Now let $x$ be a point of codimension $2$ in $X$ and let $j:B=\mathrm{Spec}(\mathcal{O}_{X,x})\hookrightarrow X$
be the corresponding open immersion. Since $X$ is smooth, $\mathcal{O}_{X,x}$
is a noetherian regular local ring of dimension $2$ and the pull-back
$\mathrm{pr}_{B}:B\times_{X}V\rightarrow B$ is an $\mathbb{A}^{n}$-fibration
restricting to a Zariski locally trivial $\mathbb{A}^{n}$-bundle
over the complement of the closed point of $B$. Furthermore, since
$j^{*}\mathcal{E}^{\vee}\simeq\mathcal{O}_{B}^{\oplus n}$, $\Omega_{B\times_{X}V/B}^{1}\simeq\mathrm{pr}_{B}^{*}j^{*}\mathcal{E}^{\vee}$
is trivial. Thus $\mathrm{pr}_{B}:B\times_{X}V\rightarrow B$ is a
trivial $\mathbb{A}^{n}$-bundle by virtue of Proposition \ref{prop:MainProp},
implying that $U$ contains all codimension $2$ points of $X$. By
descending induction on the codimension of the points of $X$, we
conclude by the same argument that $U$ contains all points of $X$. 
\end{proof}
\begin{cor}
\label{cor:Sathaye-to-DW} Assume that every $\mathbb{A}^{n}$-fibration
over the spectrum of a rank one discrete valuation ring containing
$k$ is a trivial $\mathbb{A}^{n}$-bundle. Then an $\mathbb{A}^{n}$-fibration
$\pi:V\rightarrow X$ over a smooth locally noetherian scheme $X$
is an affine-linear bundle if and only if $\Omega_{V/X}^{1}=\pi^{*}\mathcal{E}^{\vee}$
for some locally free sheaf $\mathcal{E}$ of rank $n$ on $X$.
\end{cor}
\begin{proof}
Since $X$ is smooth, for every point $x$ of codimension $1$ in
$X$, the local ring $\mathcal{O}_{X,x}$ is a rank one discrete valuation
ring. The hypothesis implies that the restriction of $\pi:V\rightarrow X$
over the image of the open embedding $\mathrm{Spec}(\mathcal{O}_{X,x})\hookrightarrow X$
is a trivial $\mathbb{A}^{n}$-bundle. The largest open subset $U$
of $X$ over which $\pi:V\rightarrow X$ restricts to a Zariski locally
trivial $\mathbb{A}^{n}$-bundle thus contains all points of codimension
$1$ of $X$, and the assertion then follows from Corollary \ref{cor:AffLin-in-Fib}. 
\end{proof}

\section{Applications }

\subsection{Dolgachev-Weisfeiler problem for $\mathbb{A}^{2}$-fibrations }
\begin{thm}
Let $V$ be a scheme on which every rank $2$ vector bundle is trivial.
Then every $\mathbb{A}^{2}$-fibration $\pi:V\rightarrow X$ over
a smooth locally noetherian scheme $X$ is a $\mathbb{G}_{a}^{2}$-torsor. 
\end{thm}
\begin{proof}
By \cite{Sa83} every $\mathbb{A}^{2}$-fibration over the spectrum
of a discrete valuation ring containing $k$ is a trivial $\mathbb{A}^{2}$-bundle.
On the other hand, since $\pi:V\rightarrow X$ is a smooth morphism
\cite[17.5.1]{EGAIV}, $\Omega_{V/X}^{1}$ is a locally free sheaf
of rank $2$ on $V$, hence is free by hypothesis. The assertion then
follows from Proposition \ref{prop:MainProp} and Corollary \ref{cor:AffLin-in-Fib}. 
\end{proof}
As a consequence of Quillen-Suslin Theorem \cite{Qu76,Su76}, we obtain:
\begin{cor}
\label{thm:Main-Thm1} An $\mathbb{A}^{2}$-fibration $\pi:\mathbb{A}^{m}\rightarrow\mathbb{A}^{n}$
is a trivial $\mathbb{A}^{2}$-bundle. 
\end{cor}
\begin{proof}
By the previous theorem $\pi:\mathbb{A}^{m}\rightarrow\mathbb{A}^{n}$
is a $\mathbb{G}_{a}^{2}$-torsor, hence is trivial since $\mathbb{A}^{n}$
is affine. 
\end{proof}
\begin{rem}
\label{thm:Main-Thm2} Combining the non-existence of nontrivial forms
of the affine plane over fields of characteristic zero \cite{Ka75}
with Lefschtez principle arguments (see e.g. \cite[Lemma 1]{KrR14}),
we get the following characterization, perhaps of more geometric nature: 

L\emph{et $k$ be an an algebraically closed field of infinite transcendence
degree over $\mathbb{Q}$. Then a morphism $\pi:\mathbb{A}_{k}^{m}\rightarrow\mathbb{A}_{k}^{n}$
whose }closed\emph{ fibers are all isomorphic to $\mathbb{A}_{k}^{2}$
is a trivial $\mathbb{A}^{2}$-bundle. }
\end{rem}

\subsection{\label{subsec:Venereau}Variables in a four dimensional polynomial
ring}

Generalizing a construction due to Vénéreau \cite{Ve01}, Daigle-Freudenburg
\cite{DaFr10} and Lewis \cite{Le13} introduced families of ``coordinate-like''
polynomials in four variables obtained as follows: given a polynomial
ring $k[x,y,z,u]$ in four variables over a field $k$ of characteristic
zero and a polynomial $p(x,y,z,u)=yu+\lambda(x,z)$ where $\lambda=z^{2}+r(x)z+s(x)$
for some polynomials $r,s\in k[x]$, we set $v=xz+yp$ and $w=x^{2}u\text{\textminus}2{\displaystyle \frac{\partial\lambda}{\partial z}}p\text{\textminus}yp^{2}$.
A \emph{Vénéreau-type polynomial} is a polynomial of the form $h=y+xQ(x,v,w)\in k[x][v,w]\subset k[x][y,z,u]$. 
\begin{example}
For $p=yu+z^{2}$ and $Q=x^{n-1}v$, one gets the famous \emph{Vénéreau
polynomials} 
\[
v_{n}=y+x^{n}(xz+y(yu+z^{2})),\quad n\geq1.
\]
The choice $p=yu+z^{2}+z$ and $Q=x^{n-1}v$ yields the family of
polynomials $b_{n}=y+x^{n}(xz+y(yu+z^{2}+z))$, $n\geq1$, which appeared
earlier in the work of Bhatwadekar and Dutta \cite[Example 4.13]{BhD94}. 
\end{example}
It is easily checked that $k[x]_{x}[y,z,u]=k[x]_{x}[h,v,w]$ and that
$k[x]/(x)[y,z,u]=k[x]/(x)[h,z,u]$. It was successively proven in
several papers \cite{Ve01,KVZ04,KaZ04,DaFr10,Le13} by clever explicit
computations that some of these Vénéreau-type polynomials $h$ are
$x$-\emph{variables} of $k[x,y,z,u]$, i.e. that there exists polynomials
$h_{2},h_{3}\in k[x][y,z,u]$ such that $k[x][h,h_{2},h_{3}]=k[x][y,z,u]$.
The fact that $v_{n}$, $n\geq3$ is an $x$-variable was established
first by Vénéreau \cite{Ve01}, and generalized for arbitrary $p=yu+\lambda(x,z)$
as above by Daigle-Freudenburg \cite{DaFr10} to all Vénéreau-type
polynomials of the form $y+x^{n}v$, $n\geq3$. The fact that $v_{2}$
is an $x$-variable was established by Lewis \cite{Le13}. In the
same article, he proved more generally that for $p=yu+z^{2}$, all
Vénéreau-type polynomials of the form $h=y+x^{2}Q(x,v,w)+x^{3}vQ_{2}(x,v^{2},w)$
are $x$-variables. 

The cases of the Vénéreau polynomial $v_{1}$ and the Bhatwadekar-Dutta
polynomial $b_{1}$ remained open so far, but since for every Vénéreau-type
polynomial $h$, the morphism $f=(x,h):\mathbb{A}_{k}^{4}\rightarrow\mathbb{A}_{k}^{2}$
is an $\mathbb{A}^{2}$-fibration \cite[Proposition 1.1.]{DaFr10},
Corollary \ref{thm:Main-Thm1} implies the following: 
\begin{prop}
\label{cor:x-variables} Every Vénéreau-type polynomial $h\in k[x,y,z,u]$
is an $x$-variable of $k[x,y,z,u]$. 
\end{prop}

\subsection{\label{subsec:Freudenburg}Locally nilpotent derivations with a slice}

Recall that given a ring $R$, an $R$-derivation $D$ of the polynomial
ring $R[x,y,z]$ in three variables over $R$ is called locally nilpotent
if $R[x,y,z]=\bigcup_{n\geq0}\mathrm{Ker}D^{n}$ . A slice for $D$
is an element $s\in R[x,y,z]$ such that $Ds=1$. The following proposition
answers a question of Freudenburg \cite{Fr09} concerning kernels
of locally nilpotent $R$-derivations of $R[x,y,z]$ with a slice.
\begin{prop}
Let $R$ be a regular ring essentially of finite type over a field
$k$ of characteristic zero. Then the kernel of a locally nilpotent
$R$-derivation of $R[x,y,z]$ with a slice is isomorphic to the symmetric
algebra $\mathrm{Sym}_{R}^{\cdot}M$ of a $1$-stably free projective
$R$-module of rank $2$.
\end{prop}
\begin{proof}
The kernel $B\subset R[x,y,z]$ of such a derivation $D$ is an $R$-algebra
of finite type such that $B[s]=R[x,y,z]$ for a slice $s\in R[x,y,z]$
of $D$. By \cite[Theorem 1.1]{Fr09}, the inclusion $R\subset B$
defines an $\mathbb{A}^{2}$-fibration $f:\mathrm{Spec}(B)\rightarrow\mathrm{Spec}(R)$.
Since $R[x,y,z]=B[s]$, $B$ is a smooth $R$-algebra and so, $\Omega_{B/R}^{1}$
is a projective $B$-module of rank $2$. The map which associates
to a finitely generated projective $B$-module $N$ the projective
$R[x,y,z]$-module $N\otimes_{B}R[x,y,z]$ is injective. On the other
hand, since $R$ is regular and essentially of finite type, it follows
from Lindell \cite{Li81} that the map which associates to a finitely
generated projective $R$-module $M$ the projective $R[x,y,z]$-module
$M\otimes_{B}R[x,y,z]$ is bijective. Consequently, there exists a
projective $R$-module $M$ of rank $2$ such that $\Omega_{B/R}^{1}\simeq B\otimes_{R}M$.
By Corollary \ref{cor:Sathaye-to-DW}, $B$ is isomorphic to $\mathrm{Sym}_{R}^{\cdot}M$.
The fact that $M$ is stably free follows from the canonical split
exact sequence 
\[
0\rightarrow\Omega_{B/R}^{1}\otimes_{B}R[x,y,z]\rightarrow\Omega_{R[x,y,z]/R}^{1}\simeq R[x,y,z]^{\oplus3}\rightarrow\Omega_{R[x,y,z]/B}^{1}\simeq R[x,y,z]\rightarrow0
\]
which implies that $\Omega_{B/R}^{1}\otimes_{B}R[x,y,z]\simeq M\otimes_{R}R[x,y,z]$
is $1$-stably free, whence that $M$ is $1$-stably free. 
\end{proof}
\begin{rem}
Note that conversely, given a $1$-stably free projective module $M$
of rank $2$ on a ring $R$, the symmetric algebra $\mathrm{Sym}_{R}^{\cdot}(M\oplus R)$
is $R$-isomophic to $R[x,y,z]$ and can be equipped via the isomorphism
$\mathrm{Sym}_{R}^{\cdot}(M\oplus R)\simeq\mathrm{Sym}_{R}^{\cdot}(M)\otimes_{R}R[s]$
with the locally nilpotent $R$-derivation ${\displaystyle D=\frac{\partial}{\partial s}}$
whose kernel is isomorphic to $\mathrm{Sym}_{R}^{\cdot}(M)$ and which
has $s$ as a slice.
\end{rem}
As a consequence of Quillen-Suslin Theorem \cite{Qu76,Su76}, we obtain
the following solution to Question 1 in \cite{Fr09}:
\begin{cor}
Let $R=k[x_{1},\ldots,x_{n}]$ be a polynomial ring in $n\geq2$ variables.
Then the kernel of a locally nilpotent $R$-derivation of $R[x,y,z]$
with a slice is isomorphic to a polynomial ring in $2$ variables
over $R$. In other words, such a derivation is conjugate to $\partial/\partial x$
by an $R$-automorphism of $R[x,y,z]$. 
\end{cor}
\bibliographystyle{amsplain}

\end{document}